\documentclass[reqno,a4paper,11pt]{article}
\usepackage{amsmath,amsthm,dsfont,amsfonts,amssymb,fancyhdr}
\usepackage[usenames,dvipsnames]{color}
\usepackage{bbm,soul,supertabular,longtable,verbatim,extarrows}
\usepackage{titlesec}
\usepackage{graphicx}
\usepackage{enumitem}
\usepackage{tikz}
\usetikzlibrary{arrows,decorations.pathmorphing,backgrounds,positioning,fit,petri}

\usepackage{caption}

\usepackage{float}
\usepackage{BOONDOX-cal}
\usepackage{cite}
\usepackage{booktabs,multirow,makecell}
\usepackage{authblk}
\usepackage[titletoc]{appendix}
\usetikzlibrary[trees]
\usetikzlibrary{arrows.meta}
\usetikzlibrary{calc}
\usepackage{lipsum}
\usepackage{mathrsfs}
\usepackage{indentfirst}
\usepackage[colorlinks=true, allcolors=blue]{hyperref}
\usepackage[top=2.4cm,bottom=2.2cm,left=2.6cm,right=2cm]{geometry}

\newtheorem{thm}{Theorem}[section]
\newtheorem{lem}[thm]{Lemma}
\newtheorem{pro}[thm]{Proposition}
\newtheorem{cor}[thm]{Corollary}
\theoremstyle{definition}
\newtheorem{de}[thm]{Definition}
\theoremstyle{remark}

\DeclareMathOperator{\Spec}{Spec}
\newcommand{\OA}{\operatorname{OA}}
\newcommand{\AG}{\operatorname{AG}}
\newcommand{\PG}{\operatorname{PG}}

\begin{document}
	\title{\bf Co-edge-regular graphs with four eigenvalues and unbounded coherent rank}
	\author[a]{Edwin R. van Dam}
	\author[a,b]{Hong-Jun Ge\footnote{Hong-Jun Ge is the corresponding author.}}
	\author[b,c]{Jack H. Koolen}
	\affil[a]{\footnotesize{Department of Econometrics and O.R., Tilburg University, Tilburg, the Netherlands}}
	\affil[b]{\footnotesize{School of Mathematical Sciences, University of Science and Technology of China, Hefei, Anhui, PR China}}
	\affil[c]{\footnotesize{Wen-Tsun Wu Key Laboratory of CAS, Hefei, Anhui, PR China}}

	\maketitle
	\pagestyle{plain}
	
	\newcommand\blfootnote[1]{%
		\begingroup
		\renewcommand\thefootnote{}\footnote{#1}%
		\addtocounter{footnote}{-1}%
		\endgroup}
	\blfootnote{E-mail addresses:  {\tt Edwin.vanDam@tilburguniversity.edu} (E.R. van Dam)  {\tt gehj22@mail.ustc.edu.cn} (H.-J. Ge) {\tt koolen@ustc.edu.cn} (J.H. Koolen).}
	
\vspace{-11pt}

\begin{abstract}
In the regular three-eigenvalue setting, spectral complexity and coherent-algebraic complexity coincide: a connected regular graph has exactly three distinct eigenvalues if and only if it is strongly regular, its coherent rank is three.
Although examples of regular graphs with four distinct eigenvalues and coherent rank larger than four are known, 
it was unknown whether coherent rank is uniformly bounded among regular graphs with four distinct eigenvalues.
	We show that no such bound exists, even under the additional assumption of co-edge-regularity.
	For every prime power \(q\), we construct infinitely many co-edge-regular graphs with exactly four distinct eigenvalues, smallest eigenvalue \(-2q-1\), and coherent rank at least \(q+4\).
	Consequently, coherent rank is unbounded among co-edge-regular graphs with exactly four distinct eigenvalues.
\end{abstract}


\section{Introduction}

Strongly regular graphs are among the central objects of algebraic graph theory; see, for example, \cite{BrVM2022}.  They lie at the intersection of finite geometry, design theory, association schemes, and spectral graph theory.  One of their most useful characterizations is spectral: a connected regular graph is strongly regular if and only if it has exactly three distinct eigenvalues. 
The coherent rank of a graph is the dimension of the coherent algebra generated by its adjacency matrix.
  From this coherent-algebraic point of view, 
  a strongly regular graph has coherent rank three.
 Thus, in the regular three-eigenvalue setting, the two natural measures of complexity---the number of distinct eigenvalues and the coherent rank---coincide in the strongest possible way.

This paper asks how far this agreement can persist beyond the strongly regular case. 
Since the adjacency algebra of a graph has dimension equal to the number of distinct eigenvalues and is contained in the coherent closure, coherent rank is always at least the number of distinct eigenvalues. 
The first regular case in which equality is no longer forced is therefore the four-eigenvalue case. 
The four-eigenvalue case is also natural from the viewpoint of association schemes: relation graphs of symmetric three-class association schemes have at most four distinct eigenvalues, and in the genuine four-eigenvalue case their coherent rank is four; see \cite{van99,vs98}. 
Thus association schemes provide a canonical rank-four model in the four-eigenvalue setting.
The question addressed here is whether coherent rank can still be uniformly bounded under strong regularity-type assumptions.
The assumption we impose is co-edge-regularity:  a regular graph is \emph{co-edge-regular} if every pair of distinct non-adjacent vertices has the same number of common neighbours.

 This co-edge-regularity assumption is natural from the viewpoint of distance-regular graphs. 
 Terwilliger showed in his lecture notes that the local graphs of thin \(Q\)-polynomial distance-regular graphs are co-edge-regular and have at most five distinct eigenvalues; see \cite{Ternotes}. 
 Thus co-edge-regular graphs with few distinct eigenvalues, and in particular with four eigenvalues, form a natural local model in the study of \(Q\)-polynomial distance-regular graphs; see also \cite{GCK2021}.

  Our main result shows that even this co-edge-regular four-eigenvalue setting admits no uniform bound on coherent rank.
  
\begin{thm}\label{thm:intro-main}
	For every integer \(q\ge2\) admitting an affine or projective plane of order \(q\), there exist infinitely many co-edge-regular graphs with exactly four distinct eigenvalues and smallest eigenvalue \(-2q-1\), for which the number of common neighbours of adjacent vertex pairs takes at least \(q+2\) distinct values. 
	In particular, these graphs have coherent rank at least \(q+4\).
\end{thm}

Consequently,  since  affine and projective planes exist for every prime power \(q\), coherent rank is unbounded among co-edge-regular  graphs with exactly four distinct eigenvalues.
Thus, even among graphs that are only one eigenvalue beyond the strongly regular case and that retain strong regularity-type conditions, coherent rank need not be bounded.

The construction has a simple geometric idea.
 We start from a regular partial linear space containing an affine or projective plane as a subspace, and we take two copies of its point set. In one layer, selected lifted lines are modified by exchanging their traces on the embedded plane, while the parts outside the plane are left unchanged. Each elementary exchange is an instance of Wang--Qiu--Hu switching~\cite{WQH2019}, and hence preserves the spectrum. The resulting graph is therefore cospectral with the corresponding untwisted two-layer graph, namely the \(2\)-clique extension of the original point graph.

Although the spectrum is preserved, the coherent structure changes substantially. The accumulated twists create many distinct common-neighbour counts on adjacent vertex pairs. These counts are detected by the coherent closure, and hence force the coherent rank to grow. In the abstract form of the construction, whenever the line size is sufficiently large relative to \(q\), the final twisted graph has at least \(q+1\) more distinct adjacent-pair common-neighbour counts than the original point graph. When the original point graph is strongly regular, this yields coherent rank at least \(q+4\).

We realize this construction in two classical design-theoretic settings.  
Nets containing affine subplanes yield twisted Latin square graphs, while regular semiplanes, equivalently duals of Steiner systems, containing projective subplanes yield twisted Steiner graphs. 
For each fixed \(q\) satisfying the corresponding plane-existence assumption, and for all sufficiently large admissible parameters, the resulting graphs have the properties stated in Theorem~\ref{thm:intro-main}.

The question of coherent rank under spectral restrictions has also been considered outside the regular setting.  Greaves and Yip~\cite{gy24} studied connected graphs with exactly three distinct eigenvalues and asked whether their coherent ranks can be unbounded.  Our result is complementary: we remain in a highly regular setting, but move from three to four eigenvalues.  
Thus four eigenvalues already suffice to separate spectral simplicity from bounded coherent rank, even under co-edge-regularity.

 The paper is organized as follows.  Section~\ref{sec:preli} collects the
 necessary graph-theoretic, coherent-algebraic, switching, and
 finite-geometric preliminaries.  Section~\ref{sec:twisted} introduces the
 twisted partial linear graphs.  Section~\ref{sec:tpl-properties} proves
 their structural and spectral properties and derives the coherent-rank lower bound from adjacent-pair common-neighbour counts.
   Section~\ref{sec:applications}
 applies the construction to nets and regular semiplanes, yielding the
 twisted Latin square and twisted Steiner graph families.  Section~\ref{sec:conclusion}
 gives concluding remarks. 
 
\section{Preliminaries}\label{sec:preli}

\subsection{Graphs}
Let $G$ be a simple graph with vertex set $V(G)$, edge set $E(G)$, and adjacency matrix $A(G)$.  Two graphs are \emph{cospectral} if their adjacency matrices have the same spectrum.  For $x\in V(G)$, write $N_G(x)=\{y\in V(G):x\sim y\}$.  For adjacent vertices $x,y$, put $\lambda_G(x,y)=|N_G(x)\cap N_G(y)|$ and define the \emph{$\lambda$-level} of $G$ by $\Lambda(G)=\#\{\lambda_G(x,y):x\sim y\}$.  For non-adjacent vertices $x,y$, put $\mu_G(x,y)=|N_G(x)\cap N_G(y)|$.  We suppress the subscript $G$ whenever the graph is clear.

A graph is \emph{sesqui-regular} with parameters $(v,k,\mu)$ if it is $k$-regular of order $v$ and every two non-adjacent vertices at distance two have exactly $\mu$ common neighbours.  It is \emph{co-edge-regular} with parameters $(v,k,\mu)$ if it is sesqui-regular with parameters $(v,k,\mu)$ and has diameter two.  A graph is \emph{strongly regular} with parameters $(v,k,\lambda,\mu)$ if it is co-edge-regular with parameters $(v,k,\mu)$ and $\lambda_G(x,y)=\lambda$ for every edge $xy$.
Equivalently, in the terminology above, strongly regular graphs are precisely co-edge-regular graphs whose $\lambda$-level is $1$.

We shall use the following standard facts.

\begin{lem}[{cf. \cite[Proposition 2]{vdh2003}}]\label{regular}
Let $\Gamma_1$ be a connected $k$-regular graph.  If $\Gamma_2$ is cospectral with $\Gamma_1$, then $\Gamma_2$ is also $k$-regular.
\end{lem}

\begin{lem}[{cf. \cite[Lemma 10.2.1]{god01}}]\label{srg}
A connected non-complete regular graph is strongly regular if and only if it has exactly three distinct eigenvalues.
\end{lem}

\subsection{Coherent rank}
Let $A=A(G)$.  The \emph{coherent closure} $W(G)$ is the smallest coherent algebra containing $A$, that is, the smallest matrix algebra over $\mathbb C$ containing $I,J,A$ that is closed under transposition and entrywise multiplication.  The \emph{coherent rank} $\operatorname{cr}(G)$ of $G$ is the dimension of $W(G)$.

\begin{lem}\label{lem:coherent-rank-lambda-level}
If $G$ is a non-complete graph with $\lambda$-level $t$, then $\operatorname{cr}(G)\ge t+2$.
\end{lem}

\begin{proof}
Let $\lambda_1,\ldots,\lambda_t$ be the distinct values of $\lambda_G(x,y)$ on edges.  For $r=0,1,\ldots,t-1$, set $M_r=A\circ(A^2)^{\circ r}$, where $(A^2)^{\circ0}=J$.  Each $M_r$ lies in $W(G)$, and on an edge with $\lambda_G(x,y)=\lambda_s$ its entry is $\lambda_s^r$.  Thus any linear relation among $M_0,\ldots,M_{t-1}$ gives a polynomial of degree at most $t-1$ vanishing at the $t$ distinct numbers $\lambda_1,\ldots,\lambda_t$, so all coefficients are zero.  Hence these $t$ matrices are linearly independent.  Together with $I$ and the non-zero matrix $J-I-A$, whose supports are disjoint from the edge support, they give $t+2$ linearly independent matrices in $W(G)$.
\end{proof}

\subsection{Clique extensions}
For a positive integer $t$, the \emph{$t$-clique extension} $G^{[t]}$ is obtained from $G$ by replacing each vertex $x$ by a clique $ X$ of order $t$, and by joining $X$ completely to $Y$ whenever $x\sim y$ in $G$.  If $G$ has $v$ vertices, then $A(G^{[t]})=J_t\otimes(A(G)+I_v)-I_{tv}$.  Consequently, if $\Spec(G)=\{[\theta_0]^{m_0},[\theta_1]^{m_1},\ldots,[\theta_r]^{m_r}\}$, then
\begin{equation}\label{cliext}
\Spec(G^{[t]})=\{[t(\theta_0+1)-1]^{m_0},[t(\theta_1+1)-1]^{m_1},\ldots,[t(\theta_r+1)-1]^{m_r},[-1]^{(t-1)v}\}.
\end{equation}
Thus cospectral graphs have cospectral $t$-clique extensions.  Also, if $G$ is co-edge-regular with parameters $(v,k,\mu)$, then $G^{[t]}$ is co-edge-regular with parameters $(tv,tk+t-1,t\mu)$.

\begin{lem}\label{lam cl ext}
If $\Lambda(G)=\gamma$, then $\Lambda(G^{[t]})\le \gamma+\#\{d_G(x):x\in V(G)\}$.  In particular, if $G$ is regular, then $\Lambda(G^{[t]})\le \Lambda(G)+1$.
\end{lem}

\begin{proof}
There are two types of edges in $G^{[t]}$.  If $x\sim y$ in $G$, then for $\widetilde x\in X$ and $\widetilde y\in Y$ we have $\lambda_{G^{[t]}}(\widetilde x,\widetilde y)=t\lambda_G(x,y)+2(t-1)$: the last term comes from $X\setminus\{\widetilde x\}$ and $ Y\setminus\{\widetilde y\}$.  If $\widetilde x\ne\widetilde x'$ lie in the same clique $ X$, then $\lambda_{G^{[t]}}(\widetilde x,\widetilde x')=td_G(x)+t-2$.  The assertion follows.
\end{proof}

\subsection{WQH-switching}\label{GM}
For $U\subseteq V(G)$ and $x\in V(G)$, write
\[
d_U(x)=|U\cap N_G(x)|.
\]
We shall use the following switching theorem of Wang, Qiu and Hu \cite{WQH2019}.

\begin{thm}[{cf. \cite[Theorem 3.5]{WQH2019}}]\label{thm:wqh-switching}
Let $\Gamma$ be a graph whose vertex set is partitioned as $V_1\cup V_2\cup W$.
Assume that the following conditions hold:
\begin{enumerate}[label=\textup{(P\arabic*)}]
\item $|V_1|=|V_2|$.
\item $d_{V_1}(u)-d_{V_2}(u)=d_{V_2}(v)-d_{V_1}(v)$ for every $u\in V_1$ and $v\in V_2$.
\item Each vertex $w\in W$ satisfies one of the following:
\begin{enumerate}[label=\textup{(\alph*)}]
\item $d_{V_1}(w)=|V_1|$ and $d_{V_2}(w)=0$;
\item $d_{V_1}(w)=0$ and $d_{V_2}(w)=|V_2|$;
\item $d_{V_1}(w)=d_{V_2}(w)$.
\end{enumerate}
\end{enumerate}
For each vertex $w\in W$ satisfying (P3)(a), delete the edges from $w$ to $V_1$ and join $w$ to every vertex of $V_2$.
For each vertex $w\in W$ satisfying (P3)(b), delete the edges from $w$ to $V_2$ and join $w$ to every vertex of $V_1$.
Keep all other adjacencies unchanged.
Then the resulting graph is cospectral with $\Gamma$.
\end{thm}

If Theorem~\ref{thm:wqh-switching} applies to $\Gamma$ with respect to $(V_1,V_2)$, we say that $\Gamma$ admits a \emph{WQH-switching} with respect to $(V_1,V_2)$.

\subsection{Partial linear spaces}

A \emph{partial linear space} is a pair $\mathcal S=(P,\mathcal L)$, where $P$ is a set of points and $\mathcal L$ is a set of subsets of $P$, called lines, such that every line contains at least two points and every two distinct points lie on at most one line.
A partial linear space $(P',\mathcal L')$ is a \emph{subspace} of $(P,\mathcal L)$ if $P'\subseteq P$ and
\[
\mathcal L'=\{L\cap P':L\in\mathcal L,\ |L\cap P'|\ge 2\}.
\]
The \emph{point graph} of $\mathcal S$ is the graph with vertex set $P$ in which two distinct points are adjacent if and only if they lie on a common line.

A partial linear space is \emph{$(s,t)$-regular} if every point is incident with $s+1$ lines and every line contains $t+1$ points.
	An \emph{affine plane of order \(q\)}, denoted by \(\AG_q\), is a \((q,q-1)\)-regular partial linear space in which for every point \(p\) and every line \(L\) with \(p\notin L\), there exists a unique line through \(p\) disjoint from \(L\).
	A \emph{projective plane of order \(q\)}, denoted by \(\PG_q\), is a \((q,q)\)-regular partial linear space in which any two distinct lines meet in exactly one point.
Throughout, an affine or projective plane of order \(q\) means an abstract finite plane of order \(q\), not necessarily Desarguesian.
By a \emph{finite plane of order $q$} we mean either $\AG_q$ or $\PG_q$.

We shall repeatedly use the following elementary consequence of regularity.

\begin{lem}\label{lem:trace-lift}
Let $\mathcal S=(P,\mathcal L)$ be a $(q,\beta)$-regular partial linear space containing the finite plane $\Pi=(P_\Pi,\mathcal G)$ of order $q$ as a subspace.
Then the following hold.
\begin{enumerate}[label=\textup{(\roman*)}]
\item For every line $g\in\mathcal G$, there is a unique line $L(g)\in\mathcal L$ such that
\[
L(g)\cap P_\Pi=g.
\]
\item No line of $\mathcal S$ meets $P_\Pi$ in exactly one point.
\item If $L\in\mathcal L$ and $L\cap P_\Pi\ne \varnothing$, then
there is a unique line $g\in \mathcal{G}$ such that $L\cap P_\Pi=g$.
\end{enumerate}
\end{lem}

\begin{proof}
The existence in (i) follows from the definition of subspace.
Uniqueness follows from the partial linear property: two distinct lines of $\mathcal S$ cannot share two points of $g$.

Let $p\in P_\Pi$.
The $q+1$ lines of $\Pi$ through $p$ have unique lifts to $\mathcal S$ by (i).
Since $\mathcal S$ is $(q,\beta)$-regular, exactly $q+1$ lines of $\mathcal S$ pass through $p$.
Hence these lifted lines are all the lines of $\mathcal S$ through $p$.
This proves (ii).
Finally, if $L\cap P_\Pi$ is non-empty, then by (ii) and by the definition of subspace it is a line of $\Pi$.
This proves (iii).
\end{proof}

\section{Twisted partial linear graphs}\label{sec:twisted}
In the construction below, we shall use a $(q,\beta)$-regular partial linear space
\(\mathcal S=(P,\mathcal L)\) containing a finite plane
\(\Pi=(P_\Pi,\mathcal G)\) of order \(q\) as a subspace.
Write
\[
\alpha=\begin{cases}
	q-1,& ~\text{if} ~\Pi\cong \AG_q,\\
	q,&~\text{if} ~\Pi\cong \PG_q.
\end{cases}
\]
Thus $\Pi$ is $(q,\alpha)$-regular and each line of $\Pi$ has $\alpha+1$ points.

Fix a line $\ell_0\in\mathcal G$.
Let
\(
\{u_1,u_2,\ldots,u_q\}\subset\ell_0.
\)

For each $a\in [q]$, let
\(
\{\ell_0,\ell_{a,1},\ell_{a,2},\ldots,\ell_{a,q}\}
\)
be the set of the $q+1$ lines of $\Pi$ through $u_a$.
When $\Pi$ is affine, we choose the labels so that, for each fixed $b\in[q]$, the lines
\(\ell_{1,b},\ell_{2,b},\ldots,\ell_{q,b}\)
are pairwise parallel.

Let $L_0=L(\ell_0)$, and for $a,b\in[q]$ let
\(
L_{a,b}=L(\ell_{a,b})
\) 
be the unique lifts in $\mathcal S$ given by Lemma~\ref{lem:trace-lift}.
For $1\le a\le \lfloor q/2\rfloor$, set
\(
\widetilde a=q+1-a.
\)
For each $b\in[q]$, the pair $(L_{a,b},L_{\widetilde a,b})$ is called a \emph{swappable pair}.
Twisting this pair means replacing these two lines by
\begin{equation}\label{eq:twist-a}
(L_{a,b}\setminus \ell_{a,b})\cup \ell_{\widetilde a,b}
\end{equation}
and
\begin{equation}\label{eq:twist-b}
(L_{\widetilde a,b}\setminus \ell_{\widetilde a,b})\cup \ell_{a,b},
\end{equation}
respectively.
All other lines are kept fixed.
In other words, the two lines keep their parts outside $P_\Pi$ unchanged,
while their traces inside $P_\Pi$ are exchanged, see the following illustrations.

\begin{figure}[H]
	\centering
	\begin{tikzpicture}[scale=1, every node/.style={font=\small}]
		
		% ---------- parameters ----------
		\def\W{6}      % width of each panel
		\def\H{4}      % height
		\def\gap{1.2}  % horizontal gap between panels
		
		% P0_\Pi box (make it larger than before)
		\def\xL{1.6}
		\def\xR{4.4}
		\def\yB{0.8}
		\def\yT{3.2}
		
		% y-coordinates for the two lines
		\def\yblue{1.5}
		\def\yred{2.5}
		
		% left/right endpoints of lines
		\def\xA{0.6}
		\def\xB{5.4}
		
		% A helper to draw the frame and P0_\Pi
		\newcommand{\DrawFrame}{
			% P0 frame
			\draw[thick] (0,0) rectangle (\W,\H);
			\node[anchor=west] at (0.1,3.7) {$P$};
			
			% P0_\Pi region (larger)
			\draw[thick] (\xL,\yB) rectangle (\xR,\yT);
			\node at (\xL+0.35,\yT-0.3) {$P_\Pi$};
		}
		
		% ================= LEFT: BEFORE =================
		\begin{scope}[xshift=0cm]
			\DrawFrame
			\node[font=\small\bfseries] at (3,4.45) {Before};
			
			% ---- L_{i,j}^{(2)} : blue everywhere (3 segments) ----
			% left outside segment
			\draw[very thick,blue] (\xA,\yblue) -- (\xL,\yblue);
			% middle segment inside P0_\Pi
			\draw[very thick,blue] (\xL,\yblue) -- (\xR,\yblue);
			% right outside segment
			\draw[very thick,blue] (\xR,\yblue) -- (\xB,\yblue);
			
			\node[blue,anchor=west] at (\xB-0.7,\yblue-0.4) {$L_{a,b}$};
			\node[blue] at (3,\yblue-0.25) {$\ell_{a,b}$};
			%		\node[black] at (0.7,0.3) {$\LL_{i,j-1}^{(2)}$};
			
			% ---- L_{q+1-i,j}^{(2)} : red everywhere (3 segments) ----
			\draw[very thick,red] (\xA,\yred) -- (\xL,\yred);
			\draw[very thick,red] (\xL,\yred) -- (\xR,\yred);
			\draw[very thick,red] (\xR,\yred) -- (\xB,\yred);
			
			\node[red,anchor=west] at (\xB-0.9,\yred+0.4) {$L_{\widetilde{a},b}$};
			\node[red] at (3,\yred+0.25) {$\ell_{\widetilde{a},b}$};
			
			% optional: tiny note
			%	\node[font=\scriptsize] at (3,0.35) {segments: outside / inside $P_0_\Pi$ / outside};
		\end{scope}
		
		% ================= RIGHT: AFTER =================
		\begin{scope}[xshift=7.2 cm]
			\DrawFrame
			\node[font=\small\bfseries] at (3,4.45) {After};
			
			% Key idea:
			% - Outside segments keep their original colors.
			% - Middle segments (inside P0_\Pi) swap colors.
			
			% ---- Line at y = yblue (still named L_{i,j}^{(2)} in labels) ----
			% outside segments remain BLUE
			\draw[very thick,blue] (\xA,\yblue) -- (\xL,\yblue);
			\draw[very thick,blue] (\xR,\yblue) -- (\xB,\yblue);
			% middle segment becomes RED (swapped)
			\draw[very thick,red]  (\xL,\yblue) -- (\xR,\yblue);
			
			\node[blue,anchor=west] at (\xB-0.7,\yblue-0.4) {$L_{a,b}$};
			\node[red] at (3,\yblue-0.25) {$\ell_{a,b}$};
			
			% ---- Line at y = yred (still named L_{q+1-i,j}^{(2)} in labels) ----
			% outside segments remain RED
			\draw[very thick,red] (\xA,\yred) -- (\xL,\yred);
			\draw[very thick,red] (\xR,\yred) -- (\xB,\yred);
			% middle segment becomes BLUE (swapped)
			\draw[very thick,blue] (\xL,\yred) -- (\xR,\yred);
			
			\node[red,anchor=west] at (\xB-0.9,\yred+0.4) {$L_{\widetilde{a},b}$};
			\node[blue] at (3,\yred+0.25) {$\ell_{\widetilde{a},b}$};

			%			\node[black] at (0.7,0.3) {$\LL_{i,j}^{(2)}$};
			
			% A simple swap indicator inside P0_\Pi
			%	\draw[-{Latex[length=2.2mm]},thick]
			%	(3.0,2.55) .. controls (3.8,2.2) and (3.8,1.8) .. (3.0,1.45);
			%	\draw[-{Latex[length=2.2mm]},thick]
			%	(3.0,1.45) .. controls (2.2,1.8) and (2.2,2.2) .. (3.0,2.55);
			%	\node[font=\scriptsize] at (4.95,2.0) {swap inside $P_0_\Pi$};
			
		\end{scope}
		% ===== Arrow from left panel to right panel =====
		% Arrow from (right edge of left box) to (left edge of right box)
		\draw[->,very thick] (\W+0.1,2.0) -- (7.2-0.1,2.0);
		\node[font=\scriptsize] at (6.6,2.25) {\eqref{eq:twist-a}-\eqref{eq:twist-b}};
	\end{tikzpicture}
	\caption{Illustration of the replacements in \eqref{eq:twist-a}--\eqref{eq:twist-b} if $P_\Pi$ is an affine plane.}
	\label{fig:swap-segment-color}
\end{figure}
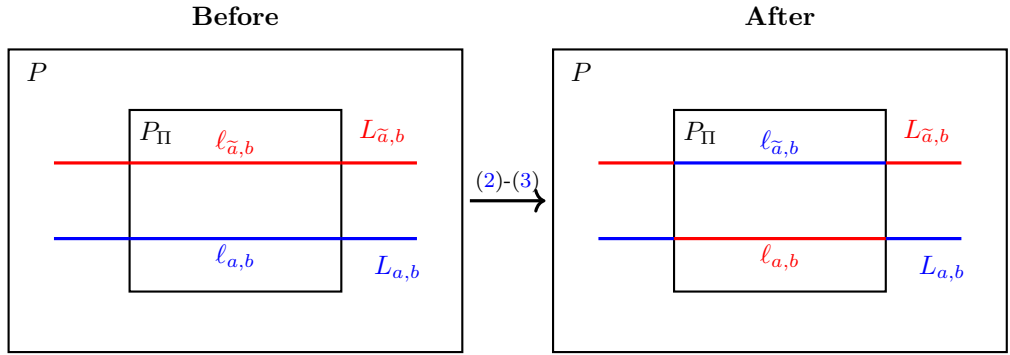

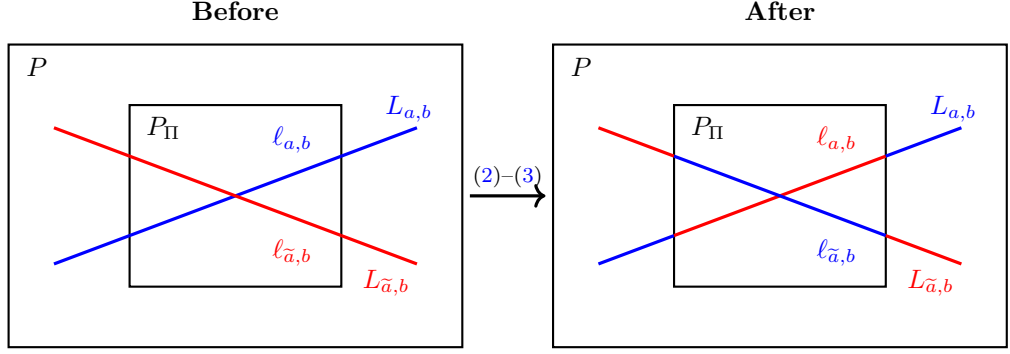
\begin{figure}[H]
	\centering
	\begin{tikzpicture}[scale=1, every node/.style={font=\small}]
		
		% ---------- parameters ----------
		\def\W{6}
		\def\H{4}
		
		% inner box P_\Pi
		\def\xL{1.6}
		\def\xR{4.4}
		\def\yB{0.8}
		\def\yT{3.2}
		
		% right panel shift
		\def\shiftR{7.2}
		
		% ---------- helper ----------
		\newcommand{\DrawFrame}{
			\draw[thick] (0,0) rectangle (\W,\H);
			\node[anchor=west] at (0.1,3.7) {$P$};
			\draw[thick] (\xL,\yB) rectangle (\xR,\yT);
			\node at (\xL+0.45,\yT-0.3) {$P_\Pi$};
		}
		
		% =========================================================
		% Coordinates for two crossing lines
		% Blue line:  (0.6,1.1) -> (5.4,2.9)
		% Red line:   (0.6,2.9) -> (5.4,1.1)
		%
		% Intersections with x = xL and x = xR:
		% blue: y = 1.475 at xL,  y = 2.525 at xR
		% red:  y = 2.525 at xL,  y = 1.475 at xR
		% =========================================================
		
		% ================= LEFT: BEFORE =================
		\begin{scope}[xshift=0cm]
			\DrawFrame
			\node[font=\small\bfseries] at (3,4.45) {Before};
			
			% blue line \ell_{a,b}
			\draw[very thick,blue] (0.6,1.1) -- (1.6,1.475);
			\draw[very thick,blue] (1.6,1.475) -- (4.4,2.525);
			\draw[very thick,blue] (4.4,2.525) -- (5.4,2.9);
			
			% red line \ell_{\tilde a,b}
			\draw[very thick,red] (0.6,2.9) -- (1.6,2.525);
			\draw[very thick,red] (1.6,2.525) -- (4.4,1.475);
			\draw[very thick,red] (4.4,1.475) -- (5.4,1.1);
			
			\node[blue,anchor=west] at (4.85,3.15) {$L_{a,b}$};
			\node[blue] at (3.75,2.75) {$\ell_{a,b}$};
			
			\node[red,anchor=west] at (4.55,0.85) {$L_{\widetilde{a},b}$};
			\node[red] at (3.75,1.25) {$\ell_{\widetilde{a},b}$};
		\end{scope}
		
		% ================= RIGHT: AFTER =================
		\begin{scope}[xshift=\shiftR cm]
			\DrawFrame
			\node[font=\small\bfseries] at (3,4.45) {After};
			
			% blue geometric line: outside remains blue, inside swapped to red
			\draw[very thick,blue] (0.6,1.1) -- (1.6,1.475);
			\draw[very thick,red]  (1.6,1.475) -- (4.4,2.525);
			\draw[very thick,blue] (4.4,2.525) -- (5.4,2.9);
			
			% red geometric line: outside remains red, inside swapped to blue
			\draw[very thick,red]  (0.6,2.9) -- (1.6,2.525);
			\draw[very thick,blue] (1.6,2.525) -- (4.4,1.475);
			\draw[very thick,red]  (4.4,1.475) -- (5.4,1.1);
			
			\node[blue,anchor=west] at (4.85,3.15) {$L_{a,b}$};
			\node[red] at (3.75,2.75) {$\ell_{a,b}$};
			
			\node[red,anchor=west] at (4.55,0.85) {$L_{\widetilde{a},b}$};
			\node[blue] at (3.75,1.25) {$\ell_{\widetilde{a},b}$};
		\end{scope}
		
		% ================= ARROW =================
		\draw[->,very thick] (\W+0.1,2.0) -- (\shiftR-0.1,2.0);
		\node[font=\scriptsize] at (6.6,2.25) {\eqref{eq:twist-a}--\eqref{eq:twist-b}};
		
	\end{tikzpicture}
	\caption{Illustration of the replacements in \eqref{eq:twist-a}--\eqref{eq:twist-b} if $P_\Pi$ is a projective plane.}
	\label{fig:swap-segment-color pp}
\end{figure}

For $1\le i\le \lfloor q/2\rfloor$ and $0\le j\le i$, define
\[
T_{i,j}:=
\{(a,b):1\le a<i,\ 1\le b\le a\}
\cup
\{(i,b):1\le b\le j\}.
\]
Also put $T_{0,0}=\varnothing$.
Starting from $\mathcal L$, twist every swappable pair indexed by $T_{i,j}$.
The resulting line set is denoted by $\mathcal L^-_{i,j}$.
For convenience, set $\mathcal L^-_{0,0}=\mathcal L$.
The order of the twists is irrelevant, since different indices correspond to different pairs of line labels.

Let
\[
P^+=\{x^+:x\in P\},\qquad P^-=\{x^-:x\in P\}.
\]
The original line set on $P^+$ is denoted by $\mathcal L^+$.
For a line label $L\in\mathcal L$, let $L^+$ be its copy in the positive layer and let $L^-$ be its realization in the negative layer after the twists indexed by $T_{i,j}$.
Set
\[
\mathcal S^+=(P^+,\mathcal L^+),\qquad
\mathcal S^-_{i,j}=(P^-,\mathcal L^-_{i,j}).
\]
In particular, 
\[
\mathcal S^+\cong \mathcal{S}\cong\mathcal S^-_{0,0},\qquad
\mathcal S^-_{i+1,0}=\mathcal{S}_{i,i}^-.
\]

\begin{de}[Twisted partial linear graph]\label{def:tpl}
For $0\le j\le i\le \lfloor q/2\rfloor$, the \emph{twisted partial linear graph}
$\Gamma_{i,j}=\Gamma_{i,j}(q,\beta)$
is the graph with vertex set
$V(\Gamma_{i,j})=P^+\cup P^-.$
Two distinct vertices $x^\varepsilon$ and $y^\delta$, where $\varepsilon,\delta\in\{+,-\}$, are adjacent if and only if at least one of the following holds:
\begin{enumerate}[label=(\roman*)]
\item $x,y\in P_\Pi$;
\item there exists a line label $L\in\mathcal L$ such that
$x^\varepsilon\in L^\varepsilon$ and $ y^\delta\in L^\delta$.
\end{enumerate}
\end{de}

The graph $\Gamma_{0,0}$ is the untwisted graph.
It is exactly the $2$-clique extension of the point graph of $\mathcal S$.

\section{Properties of twisted partial linear graphs}\label{sec:tpl-properties}

Let $G$ be the point graph of $\mathcal S$.
In this section we establish the spectral and sesqui-regular properties of the twisted partial linear graphs $\Gamma_{i,j}$.

\subsection{The twisted layer}

\begin{pro}\label{prop:twisted-pls}
For $0\le j\le i\le \lfloor q/2\rfloor$, the incidence structure
\[
\mathcal S^-_{i,j}=(P^-,\mathcal L^-_{i,j})
\]
is a $(q,\beta)$-regular partial linear space containing $\Pi$ as a subspace.
\end{pro}
\begin{proof}
	It is enough to consider a single twist, since the required assertions then follow by iteration over the twists indexed by $T_{i,j}$.
	In such a twist, the two affected lines only exchange their traces in $P_\Pi$, while their parts outside $P_\Pi$ are left unchanged.
	Hence all line sizes and all point-degrees are preserved: outside points have exactly the same incidences, and points in $P_\Pi$ merely exchange one incident lift for another.
	The set of traces on $P_\Pi$ is unchanged, so $\Pi$ remains a subspace.
	It remains to see that the partial linear property is preserved.
	For any two line labels, their outside intersection is unchanged, and their trace-intersection size is also unchanged: in the projective case any two non-empty traces meet in one point, while in the affine case each replaced trace is replaced by a parallel one, so parallelism, and hence intersection size, is preserved.
	Thus the total intersection size of any two distinct lines is the same as before the twist, and is therefore at most one.
\end{proof}

\subsection{Basic structural facts}

\begin{lem}\label{lem:basic-structure}
For $0\le j\le i\le \lfloor q/2\rfloor$, the following hold in $\Gamma_{i,j}$.
\begin{enumerate}[label=\textup{(\roman*)}]
\item The induced subgraph on $P_\Pi^+\cup P_\Pi^-$ is complete.
\item For every line label $L\in\mathcal L$,
\[
L^-\setminus P_\Pi^-=\{x^-:x\in L\setminus P_\Pi\},
\qquad
|L^-\cap P_\Pi^-|=|L\cap P_\Pi|.
\]
Moreover, for any two line labels $L_1,L_2\in\mathcal L$,
\[
|L_1^-\cap L_2^-|=|L_1\cap L_2|,
\qquad
|L_1^-\cap L_2^-\cap P_\Pi^-|=|L_1\cap L_2\cap P_\Pi|.
\]
\item If $x\in P\setminus P_\Pi$, then
\[
N_{\Gamma_{i,j}}(x^+)\cup\{x^+\}
=
N_{\Gamma_{i,j}}(x^-)\cup\{x^-\}.
\]
In particular,
\[
\bigl|N_{\Gamma_{i,j}}(x^+)\cup\{x^+\}\bigr|=2(q+1)\beta+2.
\]
\end{enumerate}
\end{lem}

\begin{proof}
Part (i) is immediate from Definition~\ref{def:tpl}.

For (ii), the outside part of each line label is never changed by a twist, and every non-empty trace is merely replaced by another line of $\Pi$ of the same size.
It remains only to compare intersections of traces.
If $\Pi$ is projective, any two non-empty traces meet in exactly one point before and after the twists, while empty traces remain empty.
If $\Pi$ is affine, a twist replaces a trace by a parallel trace; hence parallelism, and therefore the intersection number $0$ or $1$, is preserved.

For (iii), let $X_1,\dots,X_{q+1}$ be the line labels of $\mathcal S$ through $x$.
Since $x\notin P_\Pi$, part (ii) gives
\[
x^+\in X_r^+\quad\text{and}\quad x^-\in X_r^-
\qquad (1\le r\le q+1).
\]
Therefore, by Definition~\ref{def:tpl},
\[
N_{\Gamma_{i,j}}(x^+)\cup\{x^+\}
=
\bigcup_{\sigma\in\{+,-\}}\bigcup_{r=1}^{q+1}X_r^\sigma
=
N_{\Gamma_{i,j}}(x^-)\cup\{x^-\}.
\]
In each layer the $q+1$ lines through $x^\sigma$ meet pairwise only in $x^\sigma$, and each has size $\beta+1$.
Thus the displayed set has size
$2\bigl(1+(q+1)\beta\bigr)=2(q+1)\beta+2.$
\end{proof}

\begin{lem}\label{lem:regular-degree}
For $0\le j\le i\le \lfloor q/2\rfloor$, the graph $\Gamma_{i,j}$ is regular of degree
$2(q+1)\beta+1.$
\end{lem}

\begin{proof}
If $x\in P\setminus P_\Pi$, the degree follows immediately from Lemma~\ref{lem:basic-structure}~(iii).
Now let $x\in P_\Pi$.
The vertex $x^\varepsilon$ is adjacent to every vertex of $P_\Pi^+\cup P_\Pi^-$ except itself, giving $2|P_\Pi|-1$ neighbours.
For each of the $q+1$ plane lines through $x$, its lift contributes $\beta-\alpha$ vertices outside $P_\Pi$ in each layer.
Thus
\[
d_{\Gamma_{i,j}}(x^\varepsilon)=2|P_\Pi|-1+2(q+1)(\beta-\alpha).
\]
If $\Pi$ is affine, then $|P_\Pi|=q^2$ and $\alpha=q-1$; if $\Pi$ is projective, then $|P_\Pi|=q^2+q+1$ and $\alpha=q$.
In both cases the last expression equals $2(q+1)\beta+1$.
\end{proof}

\subsection{Cospectrality}

\begin{pro}\label{prop:wqh-step}
Fix $1\le i\le \lfloor q/2\rfloor$ and $0\le j<i$.
Then $\Gamma_{i,j+1}$ is obtained from $\Gamma_{i,j}$ by WQH-switching.
Consequently, $\Gamma_{i,j+1}$ and $\Gamma_{i,j}$ are cospectral.
\end{pro}

\begin{proof}
Put $r=j+1$ and write $\widetilde i=q+1-i$.
The next twist swaps the negative traces of the line labels $L_{i,r}$ and $L_{\widetilde i,r}$.
Define
\[
C_1=(\ell_{i,r}\setminus \ell_{\widetilde i,r})^-,
\qquad
C_2=(\ell_{\widetilde i,r}\setminus \ell_{i,r})^-.
\]
In both the affine and projective cases, $|C_1|=|C_2|=q$.
Let
\[
D_1^0=(L_{i,r}^+\cup L_{i,r}^-)\setminus (P_\Pi^+\cup P_\Pi^-),
\]
\[
D_2^0=(L_{\widetilde i,r}^+\cup L_{\widetilde i,r}^-)\setminus (P_\Pi^+\cup P_\Pi^-),
\]
and define the disjoint sets
\[
D_1=D_1^0\setminus D_2^0,
\qquad
D_2=D_2^0\setminus D_1^0,
\]
\[
D_3=V(\Gamma_{i,j})\setminus (C_1\cup C_2\cup D_1\cup D_2).
\]
We verify the three hypotheses of Theorem~\ref{thm:wqh-switching} for
\[
V_1=C_1,
\qquad
V_2=C_2,
\qquad
W=D_1\cup D_2\cup D_3.
\]

\smallskip
\noindent\emph{Verification of (P1).}
As noted above, $|C_1|=|C_2|=q$.
Thus (P1) holds.

\smallskip
\noindent\emph{Verification of (P2).}
By Lemma~\ref{lem:basic-structure}(i), the induced subgraph on $P_\Pi^+\cup P_\Pi^-$ is complete.
Hence the induced subgraph on $C_1\cup C_2$ is complete.
For $u\in C_1$ and $v\in C_2$,
\[
d_{C_1}(u)-d_{C_2}(u)=-1=d_{C_2}(v)-d_{C_1}(v),
\]
so (P2) holds.

\smallskip
\noindent\emph{Verification of (P3).}
We first consider vertices in $D_1$.
Let $z^\sigma\in D_1$, where $\sigma\in\{+,-\}$.
Since $z^\sigma$ lies on the line label $L_{i,r}$ in its own layer, it is adjacent to every vertex of $C_1$.
We claim that it is adjacent to no vertex of $C_2$.
Suppose, to the contrary, that $z^\sigma$ is adjacent to $y^-$ for some $y^-\in C_2$.
Then there exists a line label $M$ such that
\[
z^\sigma\in M^\sigma,
\qquad
 y^-\in M^-.
\]
Since $y^-\in P_\Pi^-$, the label $M$ has a non-empty trace in the negative layer; hence it also has a non-empty trace in the $\sigma$-layer.
Let
\[
h^\sigma=M^\sigma\cap P_\Pi^\sigma,
\qquad
h^-=M^-\cap P_\Pi^-.
\]

If \(h^\sigma\) meets \(\ell_{i,r}\), then the two lines
\(M^\sigma\) and \(L_{i,r}^\sigma\) both contain \(z^\sigma\) and also
meet in \(P_\Pi^\sigma\); by partial linearity in the \(\sigma\)-layer,
this forces \(M=L_{i,r}\).
But the pair $(L_{i,r},L_{\widetilde i,r})$ has not yet been twisted in $\Gamma_{i,j}$, so $L_{i,r}^-$ has trace $\ell_{i,r}$, which does not contain $y\in\ell_{\widetilde i,r}\setminus\ell_{i,r}$.
This is a contradiction.

If $h^\sigma$ is parallel to $\ell_{i,r}$, which means $\Pi$ is affine, then $h^-$ is also parallel to $\ell_{i,r}$, because affine twists only replace a trace by a parallel trace.
Since $h^-$ contains the point $y\in\ell_{\widetilde i,r}$ and $\ell_{\widetilde i,r}$ is the unique line through $y$ parallel to $\ell_{i,r}$, we have $h^-=\ell_{\widetilde i,r}$.
The negative-layer line with this trace is uniquely $L_{\widetilde i,r}^-$, since the pair currently being switched has not yet been twisted.
Thus $M=L_{\widetilde i,r}$, which implies $z^\sigma\in L_{\widetilde i,r}^\sigma$.
This contradicts $z^\sigma\in D_1=D_1^0\setminus D_2^0$.
Therefore every vertex of $D_1$ is adjacent to all vertices of $C_1$ and to no vertex of $C_2$.
By symmetry, every vertex of $D_2$ is adjacent to all vertices of $C_2$ and to no vertex of $C_1$.
Thus vertices in $D_1$ satisfy  (P3)(a), and vertices in $D_2$ satisfy  (P3)(b).

It remains to consider vertices in $D_3$.
Let $z^\sigma\in D_3$.
If $z\in P_\Pi$, then Lemma~\ref{lem:basic-structure}~(i) gives
\[
d_{C_1}(z^\sigma)=q=d_{C_2}(z^\sigma).
\]
Assume now that $z\in P\setminus P_\Pi$.
By Lemma~\ref{lem:basic-structure}~(iii), the numbers of neighbours of $z^+$ and $z^-$ in $C_1$ and in $C_2$ are the same.
Hence it is enough to consider $z^-$.
Let
\[
\mathcal H_z=
\{H^-\cap P_\Pi^-:z^-\in H^-,\ |H^-\cap P_\Pi^-|\ge 1\}.
\]
These are the plane traces of the negative-layer lines through $z^-$.
If $\Pi$ is projective, then $\mathcal H_z$ has at most one member: otherwise the corresponding two lines through $z^-$ would then have two common points.  
Moreover, this unique trace, if it exists, cannot be one of $\ell_{i,r}$ and $\ell_{\widetilde i,r}$ by Lemma~\ref{lem:trace-lift}~(ii), since $z^-\notin D_1\cup D_2$. 
 Such a trace meets $C_1$ and $C_2$ in equally many points.
   Hence $d_{C_1}(z^-)=d_{C_2}(z^-)$.

If $\Pi$ is affine, the traces in $\mathcal H_z$ are pairwise parallel.
Any trace not equal to $\ell_{i,r}$ or $\ell_{\widetilde i,r}$ contributes equally to $C_1$ and $C_2$: it contributes $0$ to both if it is parallel to them, and $1$ to both otherwise.
If $\ell_{i,r}\in\mathcal H_z$, then $z^-$ lies in $D_1^0$; since $z^-\in D_3$, it also lies in $D_2^0$, and therefore $\ell_{\widetilde i,r}\in\mathcal H_z$.
The two target traces together contribute $q$ neighbours in $C_1$ and $q$ neighbours in $C_2$.
The same argument applies if $\ell_{\widetilde i,r}\in\mathcal H_z$.
Thus again $d_{C_1}(z^-)=d_{C_2}(z^-)$.
Therefore every vertex in $D_3$ satisfies  (P3)(c).
This completes the verification of  (P3).

By Theorem~\ref{thm:wqh-switching}, the corresponding WQH-switching preserves the spectrum.
It remains only to identify the switched graph.
The switching changes precisely the adjacencies between $C_1\cup C_2$ and those outside vertices which lie in exactly one of the two line labels $L_{i,r}$ and $L_{\widetilde i,r}$, in either layer.
Thus the WQH-switching has exactly the effect of replacing the negative trace $\ell_{i,r}$ of $L_{i,r}^-$ by $\ell_{\widetilde i,r}$ and the negative trace $\ell_{\widetilde i,r}$ of $L_{\widetilde i,r}^-$ by $\ell_{i,r}$.
This is precisely the twist producing $\Gamma_{i,j+1}$ from $\Gamma_{i,j}$.
\end{proof}

\begin{thm}\label{thm:tpl-cospectral}
For every $0\le j\le i\le \lfloor q/2\rfloor$, the graph $\Gamma_{i,j}$ is cospectral with the $2$-clique extension $G^{[2]}$ of the point graph $G$ of $\mathcal S$.
\end{thm}

\begin{proof}
The graph $\Gamma_{0,0}$ is the untwisted two-layer graph, which is exactly $G^{[2]}$.
The line set defining any $\Gamma_{i,j}$ is obtained from the untwisted line set by a sequence of twists ordered lexicographically by the indices $(a,b)$ with $1\le b\le a\le \lfloor q/2\rfloor$.
Proposition~\ref{prop:wqh-step} shows that each single twist is realized by a WQH-switching, and hence preserves the spectrum.
Therefore every $\Gamma_{i,j}$ is cospectral with $\Gamma_{0,0}=G^{[2]}$.
\end{proof}

\subsection{Common-neighbour counts inside one positive layer}

Fix $0\le j\le i\le \lfloor q/2\rfloor$.
For $x\in P$ define
$N_-(x^+):=\{z\in P:z^-\in N_{\Gamma_{i,j}}(x^+)\}.$

\begin{pro}\label{prop:same-layer}
For any distinct \(x,y\in P\), the following hold in \(\Gamma_{i,j}\).
\begin{enumerate}[label=(\roman*)]
\item If $x\nsim y$ in $G$, then $x^+\nsim y^+$ and
$\mu_{\Gamma_{i,j}}(x^+,y^+)=2\mu_G(x,y).$
\item If $x\sim y$ in $G$, then $x^+\sim y^+$ and
$\lambda_{\Gamma_{i,j}}(x^+,y^+)=2\lambda_G(x,y)+2.$
\end{enumerate}
\end{pro}

\begin{proof}
The positive layer is a copy of the point graph $G$.
Therefore
$
x^+\sim y^+$ if and only if  $x\sim y$ in $G$,
and the positive-layer contribution to the number of common neighbours is
$|N_G(x)\cap N_G(y)|.$
We prove that the negative-layer contribution is
\begin{equation}\label{eq:minus-claim-revised}
|N_-(x^+)\cap N_-(y^+)|
=
|N_G(x)\cap N_G(y)|+2\cdot\mathbf 1_{\{x\sim y\}}.
\end{equation}

First suppose that $x,y\in P_\Pi$.
Then $x\sim y$ in $G$.
By Lemma~\ref{lem:basic-structure}(i), all vertices of $P_\Pi^-$ are adjacent to both $x^+$ and $y^+$.
Thus the contribution from $P_\Pi^-$ is $|P_\Pi|$, whereas $N_G(x)\cap N_G(y)$ contains $|P_\Pi|-2$ vertices from $P_\Pi$.
Outside $P_\Pi$, the relevant line-label intersections are unchanged by the twists, because outside parts of all line labels are unchanged.
Hence \eqref{eq:minus-claim-revised} holds.

Next suppose that $x,y\in P\setminus P_\Pi$.
In the negative layer, the lines through $x^-$ and the lines through $y^-$ have the same outside parts as the corresponding original line labels.
By Lemma~\ref{lem:basic-structure}~(ii), their pairwise intersection sizes, including their intersections inside $P_\Pi^-$, are the same as before the twist.
Thus all common neighbours other than possibly $x^-$ and $y^-$ are counted exactly as in $G$.
The vertices $x^-$ and $y^-$ are common neighbours of $x^+$ and $y^+$ precisely when $x$ and $y$ lie on a common line of $\mathcal S$, that is, precisely when $x\sim y$ in $G$.
This proves \eqref{eq:minus-claim-revised} in this case.

Finally suppose, without loss of generality, that $x\in P\setminus P_\Pi$ and $y\in P_\Pi$.
The outside common neighbours other than possibly $x^-$ are again counted as in $G$, because outside parts of line labels are unchanged by Lemma~\ref{lem:basic-structure}~(ii).
The vertex $x^-$ is a common neighbour precisely when $x\sim y$ in $G$.
For the contribution from $P_\Pi^-$, Lemma~\ref{lem:basic-structure}~(i) gives
$P_\Pi^-\subseteq N_{\Gamma_{i,j}}(y^+).$
Hence
\[
N_-(x^+)\cap N_-(y^+)\cap P_\Pi^-=N_-(x^+)\cap P_\Pi^-.
\]
The latter set is the union of the traces in $P_\Pi^-$ of the negative-layer realizations of the line labels through $x$.
By Lemma~\ref{lem:basic-structure}~(ii), its cardinality is equal to the corresponding untwisted cardinality.
In the original point graph, however, $N_G(x)\cap N_G(y)\cap P_\Pi$ excludes the point $y$, while the trace union through $x$ contains $y$ exactly when $x\sim y$.
Thus the $P_\Pi^-$ contribution is larger by $\mathbf 1_{\{x\sim y\}}$, and together with the possible contribution of $x^-$ this gives \eqref{eq:minus-claim-revised}.

Adding the positive-layer contribution proves both assertions.
\end{proof}

\begin{cor}\label{cor:sesqui}
Let $G^-$ be the point graph of the twisted partial linear space $\mathcal S^-_{i,j}$.
If both $G$ and $G^-$ are sesqui-regular with parameter $\mu>0$ and diameter $D$, then $\Gamma_{i,j}$ is sesqui-regular with parameter $2\mu$ and diameter at most $D$.
\end{cor}

\begin{proof}
Let $x^\varepsilon$ and $y^\delta$ be two non-adjacent vertices of $\Gamma_{i,j}$.
Since $P_\Pi^+\cup P_\Pi^-$ is a clique, we may assume that $y\in P\setminus P_\Pi$.
By Lemma~\ref{lem:basic-structure}~(iii), 
$\mu_{\Gamma_{i,j}}(x^\varepsilon,y^\delta)=\mu_{\Gamma_{i,j}}(x^\varepsilon,y^\varepsilon)$.
%replacing $y^\delta$ by the other copy of $y$ changes only the vertex $y^\delta$ itself, which is not a common neighbour of the pair.
Hence we may assume $\delta=\varepsilon$.

If $\varepsilon=+$, Proposition~\ref{prop:same-layer}~(i) gives
$
\mu_{\Gamma_{i,j}}(x^+,y^+)=2\mu_G(x,y)\in\{0,2\mu\}.$
If $\varepsilon=-$, the same argument applied to the point graph $G^-$ of the negative layer gives
$\mu_{\Gamma_{i,j}}(x^-,y^-)=2\mu_{G^-}(x,y)\in\{0,2\mu\}.$
Thus $\Gamma_{i,j}$ is sesqui-regular with parameter $2\mu$.

Let $d^\varepsilon$ be the distance between $x^\varepsilon$ and $y^\varepsilon$ in $G^\varepsilon$.
There is a path $x^\varepsilon z^\varepsilon_1 z^\varepsilon_{2}\cdots z^\varepsilon_{d^\varepsilon-1}y^\varepsilon$ in $G^\varepsilon$.
By Lemma~\ref{lem:basic-structure}~(iii), $z^\varepsilon_{d^\varepsilon-1}\in N_{\Gamma_{i,j}}(y^\delta)$.
Thus, there is a path of length $d^\varepsilon$ from $x^\varepsilon$ to $y^\delta$.
It follows that the diameter of $\Gamma_{i,j}$ is at most $D$.
\end{proof}

\subsection{Special common-neighbour counts across the two layers}

From now until the end of this section, fix $1\le i\le \lfloor q/2\rfloor$ and work in $\Gamma_{i,i}$.
For $1\le t\le i$, write $\widetilde t=q+1-t$.

\begin{pro}\label{prop:special-values}
Let $1\le t\le i$.
The following estimates hold in $\Gamma_{i,i}$.
\begin{enumerate}[label=\textup{(\roman*)}]
\item
$0\le
\lambda_{\Gamma_{i,i}}(u_t^+,u_t^-)
-
\bigl(2|P_\Pi|-2+2(q+1-t)(\beta-\alpha)\bigr)
\le 2t^2.$
\item
$0\le
\lambda_{\Gamma_{i,i}}(u_t^+,u_{\widetilde t}^-)
-
\bigl(2|P_\Pi|-2+2(t+1)(\beta-\alpha)\bigr)
\le 2(q-t)^2.$
\item If $w\in \ell_{1,1}\setminus\bigl(\{u_1\}\cup \ell_{q,1}\bigr),$
then
$
0\le
\lambda_{\Gamma_{i,i}}(u_1^+,w^-)-(2|P_\Pi|-2)
\le 2(q+1)^2.$
\end{enumerate}
\end{pro}

\begin{proof}
	By Lemma~\ref{lem:basic-structure}~(i), for each vertex of $P_\Pi^+\cup P_\Pi^-$, adjacency inside this set is complete.
Thus each pair considered below has the baseline contribution $2|P_\Pi|-2$ from $P_\Pi^+\cup P_\Pi^-$.
All remaining common neighbours lie outside $P_\Pi^+\cup P_\Pi^-$.

\smallskip
\noindent\emph{Proof of (i).}
The vertex $u_t^+$ uses the line labels
\[
L_0,L_{t,1},L_{t,2},\ldots,L_{t,q}.
\]
In the negative layer, after the twists indexed by $T_{i,i}$, the vertex $u_t^-$ lies on
\[
L_0,
\quad L_{t,b}\ (b>t),
\quad L_{\widetilde t,b}\ (1\le b\le t).
\]
Hence the common line labels are
\[
L_0,L_{t,t+1},L_{t,t+2},\ldots,L_{t,q},
\]
altogether $q+1-t$ labels.
Each contributes exactly $\beta-\alpha$ outside points in each layer.
This gives the lower bound
\[
2|P_\Pi|-2+2(q+1-t)(\beta-\alpha).
\]
Any additional outside common neighbour must arise from the intersection, in one of the two layers, of a label $L_{t,b}$ with $1\le b\le t$ and a label $L_{\widetilde t,c}$ with $1\le c\le t$.
There are $t^2$ such pairs, and by partial linearity each pair contributes at most one point in each layer.
Thus the number of additional common neighbours is at most $2t^2$.

\smallskip
\noindent\emph{Proof of (ii).}
The vertex $u_{\widetilde t}^-$ lies in the negative layer on
\[
L_0,
\quad L_{\widetilde t,b}\ (b>t),
\quad L_{t,b}\ (1\le b\le t).
\]
The common line labels with those used by $u_t^+$ are therefore
\[
L_0,L_{t,1},L_{t,2},\ldots,L_{t,t},
\]
altogether $t+1$ labels.
They give the lower bound
\[
2|P_\Pi|-2+2(t+1)(\beta-\alpha).
\]
Any further outside common neighbour can only come from intersections, in one of the two layers, between a label $L_{t,b}$ with $b>t$ and a label $L_{\widetilde t,c}$ with $c>t$.
There are $(q-t)^2$ such pairs, and each contributes at most one point in each layer.
This proves (ii).

\smallskip
\noindent\emph{Proof of (iii).}
The baseline contribution $2|P_\Pi|-2$ is immediate from the clique on $P_\Pi^+\cup P_\Pi^-$.
We show that no line label used by $u_1^+$ is also used by $w^-$ in the negative layer.
The labels used by $u_1^+$ are
\[
L_0,L_{1,1},L_{1,2},\ldots,L_{1,q}.
\]
After the twist of the pair $(L_{1,1},L_{q,1})$, the negative trace of $L_{1,1}$ is $\ell_{q,1}$.
In the affine case, $\ell_{1,1}$ and $\ell_{q,1}$ are parallel, so $w\notin\ell_{q,1}$.
In the projective case, $w$ was chosen outside $\ell_{q,1}$.
For $b\ge2$, the line $\ell_{1,b}$ meets $\ell_{1,1}$ only in $u_1$, and $w\ne u_1$; also $w\notin\ell_0$.
Thus none of the labels through $u_1^+$ has negative trace containing $w$.

Consequently, every outside common neighbour of $u_1^+$ and $w^-$ must arise from an intersection of one of the $q+1$ labels used by $u_1^+$ and one of the $q+1$ negative-layer labels through $w^-$.
No common label occurs, and by partial linearity each ordered pair of labels contributes at most one point in each layer.
Hence there are at most $2(q+1)^2$ additional common neighbours.
\end{proof}

\subsection{The $\lambda$-level}
The following convenient size assumption guarantees that the special \(\lambda\)-values are separated.
\begin{cor}\label{cor:tpl-level}
Assume that
$\beta\ge |P_\Pi|+(q+1)^2.$
If $\Lambda(G)=\gamma$, then
\[
\Lambda(\Gamma_{i,i})\ge \gamma+2i+2
\qquad\text{for } i<q/2.
\]
Moreover,
$\Lambda(\Gamma_{\lfloor q/2\rfloor,\lfloor q/2\rfloor})\ge \gamma+q+1.$
\end{cor}

\begin{proof}
Let
$\Lambda_0=
\{2\lambda_G(x,y)+2:x\sim y\text{ in }G\}.$
By Proposition~\ref{prop:same-layer}(ii), every value in $\Lambda_0$ occurs as a $\lambda$-value in $\Gamma_{i,i}$, and $|\Lambda_0|=\gamma$.
If $x\sim y$ in $G$, then $x$ and $y$ lie on a unique line of $\mathcal S$.
This line contributes $\beta-1$ common neighbours.
Each of the remaining $q$ lines through $x$ meets each of the remaining $q$ lines through $y$ in at most one point, so
\[
\beta-1\le \lambda_G(x,y)\le \beta-1+q^2.
\]
Therefore
\begin{equation}\label{eq:lambda0-interval}
\Lambda_0\subseteq [2\beta,2\beta+2q^2].
\end{equation}

Choose $w$ as in Proposition~\ref{prop:special-values}(iii).
Then
\[
\lambda_{\Gamma_{i,i}}(u_1^+,w^-)
\le 2|P_\Pi|-2+2(q+1)^2
<2\beta,
\]
so this value lies strictly below all values in $\Lambda_0$.

Put
$D=\beta-\alpha.$
Since $\beta\ge |P_\Pi|+(q+1)^2$ and $\alpha\le q$, we have
\begin{equation}\label{eq:D-large}
D>(q-1)^2.
\end{equation}
For $1\le t\le i$, set
\[
A_t=\lambda_{\Gamma_{i,i}}(u_t^+,u_{\widetilde t}^-),
\qquad
B_t=\lambda_{\Gamma_{i,i}}(u_t^+,u_t^-).
\]
By Proposition~\ref{prop:special-values},
\[
2|P_\Pi|-2+2(t+1)D
\le A_t
\le
2|P_\Pi|-2+2(t+1)D+2(q-t)^2,
\]
and
\[
2|P_\Pi|-2+2(q+1-t)D
\le B_t
\le
2|P_\Pi|-2+2(q+1-t)D+2t^2.
\]
Using \eqref{eq:D-large}, these estimates imply, for $i<q/2$,
\[
A_1<A_2<\cdots<A_i<B_i<B_{i-1}<\cdots<B_1.
\]
Indeed, consecutive separations follow from the inequalities
\[
D>(q-t)^2,
\qquad
D>t^2,
\qquad
D\ge \frac{(q-i)^2}{q-2i},
\]
which are all consequences of \eqref{eq:D-large} in the indicated ranges.
Moreover,
$A_1>2\beta+2q^2,$
again by the lower bound for $A_1$ and the hypothesis on $\beta$.
Thus all $A_t$ and $B_t$ are strictly above the interval \eqref{eq:lambda0-interval}.
Finally, for any $y\in P\setminus P_\Pi$, Lemma~\ref{lem:basic-structure}~(iii) gives
$$\lambda_{\Gamma_{i,i}}(y^+,y^-)=2(q+1)\beta,$$
and this is strictly larger than $B_1$.
Hence, for $i<q/2$, the graph $\Gamma_{i,i}$ has at least
\[
\gamma+1+i+i+1=\gamma+2i+2
\]
distinct $\lambda$-values.

Now assume that $q$ is even and $i=q/2$.
The same estimates give
\[
A_1<\cdots<A_{i-1}<B_i<B_{i-1}<\cdots<B_1,
\]
with all these values above $\Lambda_0$, while $\lambda_{\Gamma_{i,i}}(u_1^+,w^-)$ is below $\Lambda_0$ and $\lambda_{\Gamma_{i,i}}(y^+,y^-)$ is above $B_1$.
This gives at least
\[
\gamma+1+(i-1)+i+1=\gamma+q+1
\]
distinct $\lambda$-values.
For odd \(q\), we have \(\lfloor q/2\rfloor<q/2\), so the first part gives
\[
\Lambda(\Gamma_{\lfloor q/2\rfloor,\lfloor q/2\rfloor})
\ge \gamma+2\lfloor q/2\rfloor+2
=\gamma+q+1.
\]
The even case was just proved, so the final assertion follows.
\end{proof}

\begin{thm}\label{thm:tpl-lambda}
Assume that
$\beta\ge |P_\Pi|+(q+1)^2.$
If $\Lambda(G)=\gamma$, then
\[
\Lambda(\Gamma_{\lfloor q/2\rfloor,\lfloor q/2\rfloor})\ge \gamma+q+1.
\]
On the other hand, the $2$-clique extension $G^{[2]}$ has $\lambda$-level at most $\gamma+1$ whenever $G$ is regular.
\end{thm}

\begin{proof}
The lower bound is Corollary~\ref{cor:tpl-level}.
The upper bound for $G^{[2]}$ follows from Lemma~\ref{lam cl ext}, since $G$ is regular.
\end{proof}

\section{Applications to co-edge-regular graphs with four eigenvalues}\label{sec:applications}

We now turn the abstract construction into concrete families of graphs.
The two sources are standard finite-geometric incidence structures containing the required plane subspace: nets containing affine planes, and regular semiplanes containing projective planes.
In each case, we first record the existence input, then check that the twisted negative layer remains in the same class, and finally apply the results of Section~\ref{sec:tpl-properties}.
The resulting graphs are co-edge-regular graphs with exactly four distinct eigenvalues, smallest eigenvalue $-2q-1$, and coherent rank at least $q+4$.

\subsection{Nets and twisted Latin square graphs}\label{ls}

A \emph{net} $\mathcal{N}(q+1,n)$ of order $n$ and degree $q+1$  is a $(q,n-1)$-regular partial linear space in which, for every
point $p$ and every line $L$ with $p\notin L$, there is a unique line
through $p$ disjoint from $L$.
Equivalently, such a net is the dual incidence structure associated
with an orthogonal array $\OA(q+1,n)$, or with a transversal design
$TD(q+1,n)$; see, for example,
\cite[Chapter~X]{BJL1999}.
In particular, a net of order \(q\) and degree \(q+1\) is an affine plane
\(\AG_q\) of order \(q\). 
The point graph of such a net $\mathcal{N}(q+1,n)$ is a Latin square graph, denoted here by
\(LS_{q+1}(n)\).
It is strongly regular with parameters
\begin{equation}\label{parls}
(n^2,(n-1)(q+1),q(q-1)+n-2,q(q+1)),
\end{equation}
and spectrum
\begin{equation}\label{specls}
\Spec(LS_{q+1}(n))=
\{[(n-1)(q+1)]^1,[n-q-1]^{(n-1)(q+1)},[-q-1]^{(n-1)(n-q)}\}.
\end{equation}

Translating \cite[Theorem~1.4]{BKOT2019} under the standard
correspondence between orthogonal arrays and transversal designs, and
then passing to dual nets, gives the following form.

\begin{pro}\label{prop oa}
	Assume that there exists a net $\mathcal{N}(q+1,q)$.
	Then there exists an integer $N=N(q)$ such that, for every $n\ge N$,
	there exists a net $\mathcal{N}(q+1,n)$ containing $\mathcal{N}(q+1,q)$
	as a subspace.
\end{pro}

The next lemma verifies that the twisting operation does not leave the
class of nets.  This is the only extra point needed in order to apply the
abstract construction to Latin square graphs.

\begin{lem}\label{lem:latin-negative-layer}
Let $\mathcal S$ be the partial linear space of a net $\mathcal{N}(q+1,n)$ containing an affine plane $\AG_q$ as a subspace.
For every $0\le j\le i\le \lfloor q/2\rfloor$, the point graph of the twisted negative layer $\mathcal S^-_{i,j}$ is again a Latin square graph  $LS_{q+1}(n)$.
\end{lem}

\begin{proof}
Inside the affine subplane, the chosen swappable pairs are pairs of parallel affine lines, hence their lifts lie in the same parallel class.
A twist only exchanges the traces inside $P_\Pi$ of two lines in the same parallel class.
Therefore, for every point \(p\) and every line \(L\) with \(p\notin L\),
there is still a unique line through \(p\) disjoint from \(L\) after the twist.
Thus the twisted incidence structure is again a net $\mathcal{N}(q+1,n)$ of order $n$ and degree $q+1$.
Its point graph is therefore a Latin square graph with the parameters in \eqref{parls}.
\end{proof}

\begin{thm}\label{thm:latin-application}
Fix an integer \(q\ge2\) for which there exists an affine plane of order \(q\).  There exists an integer $N=N(q)$ such that, for every $n\ge N$, there is a graph $TLS(q,n)$ with the following properties:
\begin{enumerate}[label=\textup{(\roman*)}]
\item $TLS(q,n)$ is cospectral with the $2$-clique extension of $LS_{q+1}(n)$.
Its spectrum is
\[
\{[2(q+1)(n-1)+1]^1,
[2n-2q-1]^{(n-1)(q+1)},
[-2q-1]^{(n-1)(n-q)},
[-1]^{n^2}\}.
\]
In particular, it has exactly four distinct eigenvalues and smallest eigenvalue $-2q-1$.
\item $TLS(q,n)$ is co-edge-regular with parameter $\mu=2q(q+1)$.
\item
$
\Lambda(TLS(q,n))\ge q+2$ and $
\operatorname{cr}(TLS(q,n))\ge q+4.$
\end{enumerate}
\end{thm}

\begin{proof}
By Proposition~\ref{prop oa}, after increasing
\(N\) if necessary, there exists a net \(\mathcal N(q+1,n)\) containing
\(\AG_q\) as a subspace.  Let \(\mathcal S\) be this partial linear space,
and let \(G=LS_{q+1}(n)\) be its point graph.
Here $\beta=n-1$ and $\Pi\cong\AG_q$.
Increase the lower bound on $n$, if necessary, so that
$\beta=n-1\ge |P_\Pi|+(q+1)^2.$
Define
\[
TLS(q,n):=\Gamma_{\lfloor q/2\rfloor,\lfloor q/2\rfloor}(q,n-1).
\]

The number of vertices and the degree follow from Definition~\ref{def:tpl} and Lemma~\ref{lem:regular-degree}.
By Theorem~\ref{thm:tpl-cospectral}, $TLS(q,n)$ is cospectral with $G^{[2]}$.
Applying \eqref{cliext} to the spectrum \eqref{specls} gives the displayed spectrum.
For sufficiently large $n$, the four displayed eigenvalues are distinct, and the smallest one is $-2q-1$.

The graph $G$ is strongly regular with non-adjacent common-neighbour parameter $q(q+1)$, and by Lemma~\ref{lem:latin-negative-layer} the point graph of the negative layer has the same parameter.
Corollary~\ref{cor:sesqui} therefore gives co-edge-regularity with parameter $2q(q+1)$.
Finally, since $G$ is strongly regular, $\Lambda(G)=1$.
Theorem~\ref{thm:tpl-lambda} gives
\[
\Lambda(TLS(q,n))\ge q+2.
\]
The coherent-rank bound follows from Lemma~\ref{lem:coherent-rank-lambda-level}.
\end{proof}

\subsection{Regular semiplanes and twisted Steiner graphs}\label{st}

A \emph{regular semiplane} $\mathcal{D}(q+1,n)$ is a $(q,\frac{n-1}{q}-1)$-regular partial linear space in which any two
distinct lines meet in exactly one point.
Equivalently, such a regular semiplane is the dual incidence structure
associated with a Steiner system \(S(2,q+1,n)\); see, for example,
\cite[Section~2]{NationSeffrood2011} and \cite[Chapter~I]{BJL1999}.
In particular, a regular semiplane  $\mathcal{D}(q+1,q^2+q+1)$ is a projective plane $\PG_q$ of order $q$. 
The point graph of such a regular semiplane $\mathcal{D}(q+1,n)$ is a Steiner graph, denoted
here by \(S_{q+1}(n)\).
It is strongly regular with parameters
\begin{equation}\label{parst}
 v=\frac{n(n-1)}{q(q+1)},\quad
 k=\frac{(q+1)(n-q-1)}{q},\quad
 \lambda=q^2+\frac{n-1}{q}-2,
 \quad
 \mu=(q+1)^2,
\end{equation}
and spectrum
\begin{equation}\label{specst}
\Spec(S_{q+1}(n))=
\left\{\left[\frac{(q+1)(n-q-1)}{q}\right]^1,
\left[\frac{n-(q+1)^2}{q}\right]^{n-1},
[-q-1]^{v-n}\right\}.
\end{equation}

The congruence conditions in the next result are the usual admissibility
conditions for Steiner systems with block size \(q+1\).  The embedding
theorem of \cite[Theorem~1.6]{DLEL2016}, after passing to dual incidence
structures, gives exactly the form needed here.

\begin{pro}\label{prop:semiplane-containing-projective-plane}
	Assume that there exists a projective plane \(\PG_q\) of order \(q\).
	Then there exists an integer \(N=N(q)\) such that, for every \(n\ge N\)
	satisfying
	\[
	n-1\equiv0\pmod q,
	\qquad
	n(n-1)\equiv0\pmod{q(q+1)},
	\]
	there exists a regular semiplane \(\mathcal D(q+1,n)\) containing
	\(\PG_q\) as a subspace.
\end{pro}

We similarly need to check that twisting preserves the class of regular
semiplanes.
\begin{lem}\label{lem:steiner-negative-layer}
	Let $\mathcal S$ be the partial linear space of a regular semiplane $\mathcal{D}(q+1,n)$ containing a projective plane $\PG_q$ as a subspace.
	For every $0\le j\le i\le \lfloor q/2\rfloor$, the point graph of the twisted negative layer $\mathcal S^-_{i,j}$ is again a Steiner graph  $S_{q+1}(n)$.
\end{lem}

\begin{proof}
	By Proposition~\ref{prop:twisted-pls}, the twisted layer is a partial
	linear space, so any two distinct lines meet in at most one point.  It
	remains to show that they always meet.
	
	Consider two line labels.  If at least one of them has empty trace on the
	projective subplane, then their unique intersection before the twist lies
	outside \(P_\Pi\), and this outside intersection is unchanged by the twist.
	If both labels have non-empty traces, then after the twist their traces are
	still lines of the projective plane, and hence meet in exactly one point.
	Thus any two distinct lines of \(\mathcal S^-_{i,j}\) meet in exactly one
	point.
	
	Therefore \(\mathcal S^-_{i,j}\) is again a regular semiplane
	\(\mathcal D(q+1,n)\), and its point graph is again a Steiner graph
	\(S_{q+1}(n)\).
\end{proof}

\begin{thm}\label{thm:steiner-application}
Fix an integer \(q\ge2\) for which there exists a projective plane of order \(q\). There exists an integer $N=N(q)$ such that, for every $n\ge N$ satisfying $n-1\equiv0\pmod q$ and $n(n-1)\equiv0\pmod{q(q+1)}$, there is a graph $TS(q,n)$ with the following properties.
Let
$v=\frac{n(n-1)}{q(q+1)}.$
Then:
\begin{enumerate}[label=\textup{(\roman*)}]
\item $TS(q,n)$ is cospectral with the $2$-clique extension of $S_{q+1}(n)$.
Its spectrum is
\[
\left\{
\left[2\frac{(q+1)(n-q-1)}{q}+1\right]^1,
\left[2\frac{n-(q+1)^2}{q}+1\right]^{n-1},
[-2q-1]^{v-n},
[-1]^v
\right\}.
\]
In particular, it has exactly four distinct eigenvalues and smallest eigenvalue $-2q-1$.
\item $TS(q,n)$ is co-edge-regular with parameter $\mu=2(q+1)^2$.
\item
$
\Lambda(TS(q,n))\ge q+2$ and $
\operatorname{cr}(TS(q,n))\ge q+4.$
\end{enumerate}
\end{thm}

\begin{proof}
By Proposition~\ref{prop:semiplane-containing-projective-plane}, after
increasing \(N\) if necessary, there exists a regular semiplane
\(\mathcal D(q+1,n)\) containing \(\PG_q\) as a subspace.  Let
\(\mathcal S\) be this partial linear space, and let \(G=S_{q+1}(n)\)
be its point graph.
Here
$\beta=\frac{n-1}{q}-1$ and $
\Pi\cong\PG_q.$
Increase the lower bound on $n$, if necessary, so that
$\beta\ge |P_\Pi|+(q+1)^2.$
Define
\[
TS(q,n):=\Gamma_{\lfloor q/2\rfloor,\lfloor q/2\rfloor}\left(q,\frac{n-1}{q}-1\right).
\]

The number of vertices and the degree follow from the construction and Lemma~\ref{lem:regular-degree}.
By Theorem~\ref{thm:tpl-cospectral}, $TS(q,n)$ is cospectral with $G^{[2]}$.
Applying \eqref{cliext} to the spectrum \eqref{specst} gives the displayed spectrum.
For sufficiently large $n$, the four eigenvalues are distinct, and the smallest is $-2q-1$.

The graph $G$ is strongly regular with non-adjacent common-neighbour parameter $(q+1)^2$, and by Lemma~\ref{lem:steiner-negative-layer} the negative layer has the same sesqui-regular parameter.
Corollary~\ref{cor:sesqui} gives co-edge-regularity with parameter $2(q+1)^2$.
Finally, $\Lambda(G)=1$, so Theorem~\ref{thm:tpl-lambda} gives $\Lambda(TS(q,n))\ge q+2$, and Lemma~\ref{lem:coherent-rank-lambda-level} gives the stated coherent-rank bound.
\end{proof}

\begin{proof}[Proof of Theorem~\ref{thm:intro-main}]
	If an affine plane of order \(q\) exists, then Proposition~\ref{prop oa}
	provides, for all sufficiently large \(n\), nets \(\mathcal N(q+1,n)\)
	containing this affine plane as a subspace. Applying
	Theorem~\ref{thm:latin-application} gives infinitely many graphs with the
	claimed properties.
	
	If a projective plane of order \(q\) exists, then
	Proposition~\ref{prop:semiplane-containing-projective-plane} provides, for
	all sufficiently large admissible \(n\), regular semiplanes
	\(\mathcal D(q+1,n)\) containing this projective plane as a subspace.
	Applying Theorem~\ref{thm:steiner-application} gives infinitely many graphs
	with the claimed properties. The admissible values of \(n\) are infinite,
	since every sufficiently large \(n\equiv1\pmod{q(q+1)}\) is admissible.
\end{proof}

\section{Concluding remarks}\label{sec:conclusion}

Strongly regular graphs have coherent rank three, so in the
three-eigenvalue regular setting spectral simplicity and coherent rank coincide in the strongest possible way. 
 The construction above shows that this coincidence already breaks down for regular graphs with four distinct eigenvalues.
  For every integer \(q\ge2\) admitting a finite affine or projective plane of order \(q\), 
  the construction gives twisted Latin square or twisted Steiner graph families that are co-edge-regular, have exactly four distinct eigenvalues, smallest eigenvalue \(-2q-1\), and  coherent rank at least \(q+4\). 
Thus, coherent rank is unbounded within the class of
co-edge-regular graphs with four distinct eigenvalues.

We also point out the relation with the twisted Latin square graphs of
Ge and Koolen~\cite{GeKoolen2025}.  Their construction is global and
field-based: it uses a group-divisible orthogonal array together with a
distinguished system of affine planes in \(\mathbb F_q^3\), and hence
assumes that \(q\) is a prime power.  From the viewpoint of the present
paper, their construction can be viewed as a simultaneous version of the
same trace-exchange operation.  The underlying Latin-square net contains a
highly organized family of affine subplanes of order \(q\), and the
twisting is performed across this whole family at once.

There is also a methodological difference.  In~\cite{GeKoolen2025}, the
four-eigenvalue property is obtained by verifying structural conditions
from a combinatorial characterization of co-edge-regular graphs with at
most four distinct eigenvalues.  In the present paper, the spectral
argument is separated from the structural one: each elementary trace
exchange is a Wang--Qiu--Hu switching, so the spectrum is preserved
directly.  Consequently, the four-eigenvalue property follows from
cospectrality with the untwisted clique extension, while the coherent-rank
lower bound comes from the many adjacent-pair common-neighbour counts
created by the twists.

The present construction therefore isolates the local mechanism behind the construction of~\cite{GeKoolen2025}.  It does not use the
ambient vector space \(\mathbb F_q^3\), but only requires the underlying
partial linear space to contain a prescribed finite plane as a subspace.
Thus the prime-power assumption on \(q\) is not intrinsic to the twisting
argument itself; it enters only through the existence of the required affine or projective subplane.  
Moreover, the graphs constructed in~\cite{GeKoolen2025} have \(\lambda\)-level \(3\), which gives only the lower bound \(5\) on their coherent rank. 
By contrast, our twists are arranged to produce
coherent rank at least \(q+4\), and hence give a lower bound growing linearly
with \(q\).

Finally, the semiplane construction gives an affirmative answer to the
Steiner-graph analogue raised in~\cite[Problem~34(2)]{GeKoolen2025}.
Indeed, applying the same trace-exchange mechanism to regular semiplanes
containing projective subplanes yields the twisted Steiner graph families
constructed above.

The construction leaves several natural questions open.  For instance,
our lower bound on coherent rank comes only from adjacent-pair
common-neighbour counts, so it would be interesting to determine the exact
coherent ranks of these twisted graphs.  It is also natural to ask whether
one can obtain unbounded coherent rank in the four-eigenvalue
co-edge-regular setting while keeping the smallest eigenvalue fixed.

\section*{Acknowledgements}
 H.-J. Ge is supported by CSC scholarship program (No.  202506340038).
J.H. Koolen is partially supported by the National Key R. and D. Program of China (No. 2020YFA0713100), the National Natural Science Foundation of China 
(No. 12071454, No. 12471335), and the Anhui Initiative in Quantum Inform
ation Technologies (No. AHY150000).

\bibliographystyle{plain}
\bibliography{GVK2026}

@article{NationSeffrood2011,
  author  = {Nation, J. B. and Seffrood, J. Y. G.},
  title   = {Dual linear spaces generated by a non-{D}esarguesian configuration},
  journal = {Contrib. Discrete Math.},
  volume  = {6},
  number  = {1},
  pages   = {98--141},
  year    = {2011},
  doi     = {10.55016/ojs/cdm.v6i1.62078}
}

@article{GCK2021,
  author  = {Gebremichel, B. and Cao, M.-Y. and Koolen, J. H.},
  title   = {Two characterizations of the grid graphs},
  journal = {Discrete Math.},
  volume  = {344},
  number  = {11},
  pages   = {112550},
  year    = {2021},
  doi     = {10.1016/j.disc.2021.112550}
}

@book{BJL1999,
  author    = {Beth, T. and Jungnickel, D. and Lenz, H.},
  title     = {Design Theory},
  edition   = {Second},
  publisher = {Cambridge University Press},
  address   = {Cambridge},
  year      = {1999},
  volumes   = {2}
}

@article{BKOT2019,
  author  = {Barber, B. and K{\"u}hn, D. and Lo, A. and Osthus, D. and Taylor, A.},
  title   = {Clique decompositions of multipartite graphs and completion of {L}atin squares},
  journal = {J. Combin. Theory Ser. A},
  volume  = {151},
  pages   = {146--201},
  year    = {2017},
  doi     = {10.1016/j.jcta.2017.04.005}
}

@article{DLEL2016,
  author  = {Dukes, P. and Lamken, E. R. and Ling, A. C. H.},
  title   = {An existence theory for incomplete designs},
  journal = {Canad. Math. Bull.},
  volume  = {59},
  number  = {2},
  pages   = {287--302},
  year    = {2016},
  doi     = {10.4153/CMB-2015-073-7}
}

@misc{gy24,
  author       = {Greaves, G. and Yip, J.},
  title        = {The coherent rank of a graph with three eigenvalues},
  year         = {2024},
  howpublished = {arXiv:2406.17395}
}

@article{WQH2019,
  author  = {Wang, W. and Qiu, L. and Hu, Y.},
  title   = {Cospectral graphs, {GM}-switching and regular rational orthogonal matrices of level {$p$}},
  journal = {Linear Algebra Appl.},
  volume  = {563},
  pages   = {154--177},
  year    = {2019},
  doi     = {10.1016/j.laa.2018.10.027}
}

@article{GeKoolen2025,
  author  = {Ge, H.-J. and Koolen, J. H.},
  title   = {On co-edge-regular graphs with {4} distinct eigenvalues},
  journal = {Electron. J. Combin.},
  volume  = {32},
  number  = {3},
  pages   = {{\#}P3.55},
  year    = {2025},
  doi     = {10.37236/13958}
}

@article{van99,
  author  = {van Dam, E. R.},
  title   = {Three-class association schemes},
  journal = {J. Algebraic Combin.},
  volume  = {10},
  number  = {1},
  pages   = {69--107},
  year    = {1999},
  doi     = {10.1023/A:1018628204156}
}

@article{vs98,
  author  = {van Dam, E. R. and Spence, E.},
  title   = {Small regular graphs with four eigenvalues},
  journal = {Discrete Math.},
  volume  = {189},
  number  = {1--3},
  pages   = {233--257},
  year    = {1998},
  doi     = {10.1016/S0012-365X(98)00085-5}
}

@article{vdh2003,
  author  = {van Dam, E. R. and Haemers, W. H.},
  title   = {Which graphs are determined by their spectrum?},
  journal = {Linear Algebra Appl.},
  volume  = {373},
  pages   = {241--272},
  year    = {2003},
  doi     = {10.1016/S0024-3795(03)00483-X}
}

@book{god01,
  author    = {Godsil, C. D. and Royle, G. F.},
  title     = {Algebraic Graph Theory},
  series    = {Graduate Texts in Mathematics},
  volume    = {207},
  publisher = {Springer},
  address   = {New York},
  year      = {2001},
  doi       = {10.1007/978-1-4613-0163-9}
}

@misc{Ternotes,
  author       = {Terwilliger, P.},
  title        = {Lecture Notes on {T}erwilliger Algebra},
  note         = {Edited by H. Suzuki},
  year         = {2023},
  howpublished = {\url{https://icu-hsuzuki.github.io/t-algebra/t-algebra.pdf}}
}

@book{BrVM2022,
  author    = {Brouwer, A. E. and Van Maldeghem, H.},
  title     = {Strongly Regular Graphs},
  publisher = {Cambridge University Press},
  address   = {Cambridge},
  year      = {2022},
  doi       = {10.1017/9781009057226}
}

\end{document}